\documentclass[12pt]{amsart}
\usepackage{amsmath, amsthm, amscd, amsfonts, amssymb, graphicx, color}
\usepackage[pagebackref=false,colorlinks,linkcolor=blue,citecolor=red]{hyperref}
\usepackage{accents}
\usepackage[all]{xy}

\usepackage{pgf,tikz,float}
\usepackage{mathrsfs}
\usetikzlibrary{arrows}

\def\frk{\frak}               

\def\Phi{{\frk n}}
\def\Phi{{\frk N}}
\def\KK{{\mathbb K}}
\def\opn#1#2{\def#1{\operatorname{#2}}} 
\opn\chara{char}
\opn\length{\ell}
\opn\pd{pd}
\opn\rk{rk}
\opn\projdim{proj\,dim}
\opn\cohdim{cd}
\opn\injdim{inj\,dim}
\opn\rank{rank}
\opn\depth{depth}
\opn\grade{grade}
\opn\height{height}
\opn\embdim{emb\,dim}
\opn\codim{codim}

\opn\Tr{Tr}
\opn\bigrank{big\,rank}
\opn\superheight{superheight}
\opn\lcm{lcm}
\opn\trdeg{tr\,deg}
\opn\reg{reg}
\opn\hilb{Hilb}
\opn\hpolynomial{h}
\opn\cdeg{cdeg}
\opn\lreg{lreg}
\opn\ini{in}
\opn\lpd{lpd}
\opn\size{size}
\opn\bigsize{bigsize}
\opn\cosize{cosize}
\opn\bigcosize{bigcosize}
\opn\sdepth{sdepth}
\opn\sreg{sreg}
\opn\link{link}
\opn\fdepth{fdepth}
\opn\lin{lin}
\opn\ini{in}
\opn\div{div}
\opn\Div{Div}
\opn\cl{cl}
\opn\Cl{Cl}
\opn\Spec{Spec}
\opn\Supp{Supp}
\opn\supp{supp}
\opn\Sing{Sing}
\opn\Ass{Ass}
\opn\Min{Min}
\opn\Mon{Mon}
\opn\dstab{dstab}
\opn\astab{astab}
\opn\Syz{Syz}
\opn\Ann{Ann}
\opn\Rad{Rad}
\opn\Soc{Soc}
\opn\Im{Im}
\opn\Ker{Ker}
\opn\Coker{Coker}
\opn\Am{Am}
\opn\Hom{Hom}
\opn\Tor{Tor}
\opn\Ext{Ext}
\opn\End{End}
\opn\Aut{Aut}
\opn\id{id}

\opn\nat{nat}
\opn\pff{pf}
\opn\Pf{Pf}
\opn\GL{GL}
\opn\SL{SL}
\opn\mod{mod}
\opn\ord{ord}
\opn\Gin{Gin}
\opn\Hilb{Hilb}
\opn\sort{sort}
\opn\initial{init}
\opn\ende{end}
\opn\height{height}
\opn\type{type}
\opn\mdeg{mdeg}
\opn\aff{aff}
\opn\con{conv}
\opn\relint{relint}
\opn\st{st}
\opn\lk{lk}
\opn\cn{cn}
\opn\core{core}
\opn\vol{vol}
\opn\link{link}
\opn\lex{lex}
\opn\sign{sign}
\opn\gr{gr}

\def\pot#1#2{#1[\kern-0.28ex[#2]\kern-0.28ex]}

\opn\dirlim{\underrightarrow{\lim}}
\opn\inivlim{\underleftarrow{\lim}}
\def\Implies{\ifmmode\Longrightarrow \else
	\unskip${}\Longrightarrow{}$\ignorespaces\fi}
\def\implies{\ifmmode\Rightarrow \else
	\unskip${}\Rightarrow{}$\ignorespaces\fi}
\def\iff{\ifmmode\Longleftrightarrow \else
	\unskip${}\Longleftrightarrow{}$\ignorespaces\fi}

\let\:=\colon
\newtheorem{Theorem}{Theorem}[section]
\newtheorem{Lemma}[Theorem]{Lemma}
\newtheorem{Corollary}[Theorem]{Corollary}
\newtheorem{Proposition}[Theorem]{Proposition}
\newtheorem{Remark}[Theorem]{Remark}

\newtheorem{Example}[Theorem]{Example}

\newtheorem{Definition}[Theorem]{Definition}

\let\epsilon\varepsilon
\def\pnt{{\raise0.5mm\hbox{\large\bf.}}}

\begin{document}
	\title{binomial edge ideals of small depth}
	\author {M. Rouzbahani Malayeri, S. Saeedi Madani, D. Kiani}
		
	\address{Mohammad Rouzbahani Malayeri, Department of Mathematics and Computer Science, Amirkabir University of Technology (Tehran Polytechnic), Tehran, Iran}
	\email{m.malayeri@aut.ac.ir}
	
		\address{Sara Saeedi Madani, Department of Mathematics and Computer Science, Amirkabir University of Technology (Tehran Polytechnic), Tehran, Iran, and School of Mathematics, Institute for Research in Fundamental Sciences (IPM), Tehran, Iran}
	\email{sarasaeedi@aut.ac.ir}
	
		\address{Dariush Kiani, Department of Mathematics and Computer Science, Amirkabir University of Technology (Tehran Polytechnic), Tehran, Iran, and School of Mathematics, Institute for Research in Fundamental Sciences (IPM), Tehran, Iran}
	\email{dkiani@aut.ac.ir}

	\begin{abstract}
Let $G$ be a graph on $[n]$ and $J_G$ be the binomial edge ideal of $G$ in the polynomial ring $S=\KK[x_1,\ldots,x_n,y_1,\ldots,y_n]$. In this paper we investigate some topological properties of a poset associated to the minimal primary decomposition of $J_G$. We show that this poset admits some specific subposets which are contractible. This in turn, provides some interesting algebraic consequences. In particular, we characterize all graphs $G$ for which $\depth S/J_G=4$.

	\end{abstract}

	
	
	\maketitle
	
\section{Introduction}\label{intro}
Binomial edge ideals were introduced in $2010$ by Herzog, Hibi, Hreinsd{\'o}ttir, Kahle and Rauh in \cite{HHHKR} and independently by Ohtani in \cite{O}. Let $G$ be a simple graph on the vertex set $[n]$ and the edge set $E(G)$. Let $S=\KK[x_1,\ldots,x_n,y_1,\ldots,y_n]$ be the polynomial ring over $2n$ variables where $\KK$ is a field. Then the \textit{binomial edge ideal} associated to the graph $G$ denoted by $J_G$ is an ideal in $S$ which is defined as follows:
\[
J_G=(x_iy_j-x_jy_i: \{i,j\}\in E(G), 1\leq i<j\leq n).
\] 
This class of ideals could be interpreted as a natural generalization of the well-studied so-called \textit{determinantal ideal} of the $(2\times n)$-matrix
\[X=\left[
\begin{array}{ccc}
x_1 & \cdots & x_n \\
y_1 & \cdots & y_n \\
\end{array}
\right].
\]	
\par The study of algebraic properties as well as numerical invariants of binomial edge ideals has attracted a considerable attention in the meantime, see e.g. \cite{A, BN, BMS, EHH, ERT, EZ, KS, KS2, KumarS, MM, RSK, SK, SK2}.
\par One of the interesting homological invariants associated to binomial edge ideals is depth. Let $H_{\mathfrak{m}}^{i}(S/J_G)$ denote the $i^{th}$ local cohomology module of $S/J_G$ supported on the irrelevant maximal ideal $\mathfrak{m}=(x_1,\ldots,x_n,y_1,\ldots,y_n)$. Then we have
\[
\depth S/J_G=\min \{i:H_{\mathfrak{m}}^{i}(S/J_G)\neq (0)\}.
\]
\par While computing the depth of the binomial edge ideal of a graph is hard in general, there have been several attempts to get some interesting results in this direction for some special families of graphs. Moreover, some lower and upper bounds for the depth of binomial edge ideal of graphs have been obtained by several authors which will be briefly discussed in the sequel.
\par Let $C_n$ denote the cycle on $n$ vertices. In \cite{ZZ} it was shown that $\depth S/J_{C_{n}}=n$, for $n>3$. Also, in \cite{EHH} the authors showed that $\depth S/J_G=n+1$, for a connected \textit{block} graph~$G$. Later in \cite{KS1} the authors computed the depth of a wider class of graphs which are called \textit{generalized block} graphs. In \cite{KumarS}, a nice formula was given for the depth of the join product of two graphs $G_1$ and $G_2$. Roughly speaking, the join product of two graphs $G_1$ and $G_2$, denoted by $G_1\ast G_2$, is the graph which is obtained from the union of $G_1$ and $G_2$ by joining all the vertices of $G_1$ to vertices of $G_2$, (the precise definition is given in Section 2). 
\par In \cite{BN} the authors gave an upper bound for the depth of the  binomial edge ideal of a graph in terms of some graphical invariants. Indeed, they showed that for a non-complete connected graph $G$, $\depth S/J_G\leq n-\kappa(G)+2$, where $\kappa(G)$ denotes the vertex connectivity of $G$. 
\par There is also a lower bound for $\depth S/J_G$. Indeed, let $\cohdim(J_G,S)$ denote the cohomological dimension of $J_G$ in $S$, namely,  $\cohdim(J_G,S)=\max\{i:H^{i}_{J_G}(S)\neq (0)\}$. Now a result in \cite{F}  due to Faltings implies that
\begin{equation}\label{falting}
\cohdim(J_G,S)\leq 2n-\Bigl\lfloor \dfrac{2n-1}{\mathrm{bigheight}\hspace{0.9mm}J_G}\Bigr\rfloor,
\end{equation}
whenever $J_G\neq (0)$. On the other hand, since $S/J_G$ is a cohomologically full ring by a result in \cite{CV}, (see \cite{DDM} for the definition of cohomologically full rings),   by \cite{DDM} we have that  
\begin{equation}\label{dao}
\depth S/J_G\geq 2n-\cohdim(J_G,S).
\end{equation}
Hence, by \eqref{falting} and \eqref{dao}, we get the following lower bound for the depth of $S/J_G$:
\begin{equation}\label{general}
\depth S/J_G\geq \Bigl\lfloor \dfrac{2n-1}{\mathrm{bigheight}\hspace{0.9mm}J_G}\Bigr\rfloor .
\end{equation}
\par In this paper we apply some results and techniques from the topology of posets to study the depth of binomial edge ideals. We are interested in studying binomial edge ideals of small depth. Specifically, we characterize all graphs $G$ with $\depth S/J_G=4$. This is based on a Hochster type decomposition formula for the local cohomology modules of binomial edge ideals provided recently by \`Alvarez Montaner in \cite{A}.
\par This paper is organized as follows. In Section \ref{per}, we fix the notation and review some definitions and some known facts that will be used throughout the paper.
\par In Section \ref{hochster type}, we associate a poset to the binomial edge ideal of a graph. Then, we state in Theorem \ref{Hochster} the Hochster type formula for the local cohomology modules of binomial edge ideals arised from \cite[Theorem~3.9]{A}.  
\par Section \ref{topology} is devoted to extract some topological properties like contractibility of some specific subposets of the poset associated to binomial edge ideals which is introduced in Definition \ref{typo}.
\par In Section \ref{char}, in Theorem \ref{lower}, we supply a lower bound for the depth of binomial edge ideals. In particular, we show that   
$\depth S/J_G\geq 4$, where $G$ is a graph with at least three vertices. Then, by using the aforementioned lower bound and also by the provided ingredients in Section \ref{topology}, we characterize all graphs $G$ with $\depth S/J_G=4$, in Theorem \ref{small}.

\section{Preliminaries}\label{per}
In this section we review some notions and facts that will be used throughout the paper. In this paper all graphs are assumed to be simple (i.e. with no loops, directed and multiple edges).
\par \medskip Let $G$ be a graph on $[n]$ and $T\subseteq [n]$. A subgraph $H$ of $G$ on the vertex set $T$ is called an \emph{induced subgraph} of $G$, whenever for any two vertices $i,j\in T$ such that $\{i,j\}\in E(G)$, one has $\{i,j\}\in E(H)$. Moreover, by $G-T$, we mean the induced subgraph of $G$ on the vertex set $[n]\backslash T$. A vertex $i\in [n]$ is said to be a \emph{cut vertex} of $G$ whenever $G-\{i\}$ has more connected components than $G$. We say that $T$ has \emph{cut point property} for $G$, whenever each $i\in T$ is a cut vertex of the graph $G-(T\backslash \{i\})$. In particular,  the empty set $\emptyset$, has cut point property for $G$. We denote by $\mathcal{C}(G)$, the family of all subsets $T$ of $[n]$ which have the cut point property for $G$. Namely,
\[
\mathcal{C}(G)=\{T\subseteq [n]:T~\mathrm{has~cut~point~property~for}~G\}.
\]
\par Let $G_1$ and $G_2$ be two graphs on the disjoint vertex sets $V(G_1)$ and $V(G_2)$, respectively. Then by the \emph{join product} of $G_1$ and $G_2$, denoted by $G_1*G_2$, we mean the graph on the vertex set $V(G_1)\cup V(G_2)$ and the edge set 
\[
E(G_1)\cup E(G_2)\cup \{\{u,v\}: u\in V(G_1)~\mathrm{and}~v\in V(G_2)\}.
\]
\par Let $G$ be a graph and $T\subseteq [n]$. Assume that $G_1 , \ldots , G_{c_G(T)}$ are the connected components of $G-T$. Let $\widetilde{G}_1 , \ldots , \widetilde{G}_{c_G(T)}$ be the complete graphs on the vertex sets $V(G_1), \ldots , V(G_{c_G(T)})$, respectively, and let
\[
P_T(G)=(x_i,y_i)_{i\in T} +J_{\widetilde{G}_1}+\cdots +J_{\widetilde{G}_{c_G(T)}}.
\]
Then it is easily seen that
\[
\mathrm{height}\hspace{0.9mm}P_T(G)=n-c_G(T)+|T|.
\]
Also, in \cite[Corollary~3.9]{HHHKR}, it was shown that $P_T(G)$ is a minimal prime ideal of $J_G$ if $T$ has cut point property for $G$. Moreover, it was proved in \cite[Corollary~2.2]{HHHKR} that $J_G$ is a radical ideal. So, $J_{G}=\bigcap\limits_{\substack{T\in \mathcal{C}(G)}}P_T(G)$.
\par \medskip Let $\Delta$ be a simplicial complex. Recall that the \emph{$1$-skeleton} graph of $\Delta$ is the subcomplex of $\Delta$ consisting of all of the faces of $\Delta$ which have cardinality at most $2$. The simplicial complex $\Delta$ is said to be \emph{connected} if its $1$-skeleton graph is connected.
\par \medskip Let $(\mathcal{P},\preccurlyeq)$ be a poset. Recall that the \emph{order complex} of $\mathcal{P}$, denoted by $\Delta(\mathcal{P})$, is the simplicial complex whose facets are the maximal chains in $\mathcal{P}$. If $\mathcal{P}$ is an empty poset, then we consider $\Delta(\mathcal{P})=\{\emptyset\}$, i.e. the empty simplicial complex.

\section{Hochster type formula}\label{hochster type}
In this section we focus on a Hochster type formula for the local cohomology modules of binomial edge ideals recently provided by \`Alvarez Montaner in \cite{A}. First we need to recall the definition of a poset associated to the binomial edge ideal of a graph $G$ from \cite{A}. 
\par Let $I$ be an ideal in the polynomial ring $S$ and $I=q_1\cap \cdots \cap q_t$ be a not necessarily minimal decomposition for the ideal $I$. Now, the set of all possible sums of ideals in this decomposition forms a poset ordered by the reverse inclusion and is denoted by $\mathcal{P}_I$. In the special case, we use the notation $\mathcal{P}_{G}$, instead of $\mathcal{P}_{J_G}$, for the poset arised from the minimal primary decomposition of $J_G$.
\par Now, the following, is the definition of a poset associated to the binomial edge ideal of a graph $G$ which was introduced in \cite[Definition~3.3]{A}. 
\par Let $G$ be a graph. Associated to $J_G$ is the following poset which is denoted by $\mathcal{A}_{G}$: The ideals contained  in $\mathcal{A}_{G}$ are the prime ideals in $\mathcal{P}_{G}$, the prime ideals in the posets $\mathcal{P}_{I}$ arised from the minimal primary decomposition of every non prime ideal $I$ in $\mathcal{P}_{G}$ and the prime ideals that are obtained by repeating this procedure every time a non prime ideal is discovered.  
\par \medskip Note that for some technical goals that will be discussed later, we need to consider another poset associated to binomial edge ideals. Indeed, the importance of our new poset will be exhibited when we study the topological properties of some specific subposets of it, see Lemma \ref{vital}, Theorem~\ref{north} and also Remark \ref{remark}.
\par Now, inspired by the \`Alvarez Montaner's definition, we define a new poset associated to the binomial edge ideal of a graph $G$ as follows: 
\begin{Definition}\label{typo}
\em{Let $G$ be a graph on $[n]$ and $J_{G}=\bigcap\limits_{\substack{T\in \mathcal{C}(G)}}P_T(G)$ be the minimal primary decomposition of $J_G$.  Associated to this decomposition, we consider the poset $(\mathcal{Q}_G,\preccurlyeq)$ ordered by reverse inclusion which is made up of the following elements:
\begin{itemize}
\item
the prime ideals in the poset $\mathcal{P}_{G}$,
\item 
the prime ideals in the posets $\mathcal{P}_{I}$, arised from the following type of decompositions:
\[
I=q_1\cap q_2 \cap \cdots \cap q_t \cap
(q_1+P_\emptyset(G)) \cap (q_2+P_\emptyset(G)) \cap \cdots \cap (q_t+P_\emptyset(G)),\] 
where $I$'s are the non-prime ideals in the poset $\mathcal{P}_{G}$ and 
$q_1, q_2,\ldots, q_t$ are the minimal prime ideals of $I$, and
 \item
the prime ideals that we obtain repeatedly by this procedure every time that we find a non-prime ideal.
\end{itemize}
}
\end{Definition}
It is worth mentioning here that the process of the construction of our poset $\mathcal{Q}_G$ terminates after a finite number of steps just like the construction process of the poset $\mathcal{A}_{G}$. Indeed, as we will see in Corollary~\ref{crucial}, every element $q$ in the poset $\mathcal{Q}_G$ is of the form $P_T(H)$ for some graph $H$ on the vertex set $[n]$ and some $T\subseteq [n]$.
\par Moreover, the main difference between our construction of the poset $\mathcal{Q}_G$ and the construction of the poset $\mathcal{A}_{G}$, is that our construction involves a special decomposition of the form $I=q_1\cap q_2 \cap \cdots \cap q_t \cap
(q_1+P_\emptyset(G)) \cap (q_2+P_\emptyset(G)) \cap \cdots \cap (q_t+P_\emptyset(G))$ for the non-prime ideals $I$, in spite of using the minimal primary decomposition for the ideals $I$ in the construction of $\mathcal{A}_{G}$. It turns out that $\mathcal{A}_{G}$ is a subposet of the poset $\mathcal{Q}_G$. 
 Now, a natural question might be whether the posets $\mathcal{Q}_G$ and $\mathcal{A}_{G}$ coincide in general or not. The following example aims to show that, in general, $\mathcal{Q}_G$ and 
$\mathcal{A}_{G}$ do not coincide. 
\newpage Here, $K_n$ denotes the complete graph on $[n]$.
\begin{Example}
\em{Let $G$ be the graph shown in Figure \ref{Sa}. One could see that
\[\mathcal{C}(G)=\{T: T\subseteq \{2,4,7,9\}\}.\]
Now, since  
\[(x_2,y_2,x_3,y_3,x_4,y_4,x_7,y_7,x_8,y_8,x_9,y_9)\in \Min(P_{\{2,9\}}(G)+P_{\{4,7\}}(G)),\] we have that $q=(x_2,y_2,x_3,y_3,x_4,y_4,x_7,y_7,x_8,y_8,x_9,y_9)+J_{K_{10}-\{2,3,4,7,8,9\}}\in \mathcal{Q}_G$. However, $q\notin \mathcal{A}_{G}$. Indeed, assume on contrary that $q\in \mathcal{A}_{G}$. Now let $A=\{P_{T_1}(G),\ldots,P_{T_s}(G)\}$ be a subset of the maximal elements of the poset $\mathcal{A}_{G}$ with the property that $I=\sum_{i=1}^{s}P_{T_i}(G)$ creates the element $q$ in the process of construction of the poset $\mathcal{A}_{G}$. We call the set $A$ with such property, a \emph{predecessor} for the element $q$. Now, regarding the described  structure of the minimal prime ideals of $J_G$, and also by using the fact that the vertices $5$ and $10$ are adjacent in the graph $K_{10}-\{2,3,4,7,8,9\}$, one could easily check that $P_{\{4\}}(G)\in A$ and  $P_{\{9\}}(G)\in A$. This implies that 
\[I=(x_j,y_j:j\in \cup_{i=1}^{s}T_i)+J_L,\]
where $L$ is the join product of two isolated vertices $5$ and $10$, with the complete graph on the vertex set $\{1,2,3,6,7,8\}$. So, $I=q_1\cap q_2$ is the minimal primary decomposition for $I$, where
\[q_1=(x_j,y_j:j\in \cup_{i=1}^{s}T_i)+(x_k,y_k:k\in \{1,2,3,6,7,8\})\]
and
\[q_2=(x_j,y_j:j\in \cup_{i=1}^{s}T_i)+P_\emptyset(L). \]
Now, since $q_1+q_2$ is a prime ideal and $q\notin\{q_1,q_2,q_1+q_2\}$, we get a contradiction with the fact that $A$ is a predecessor for $q$.
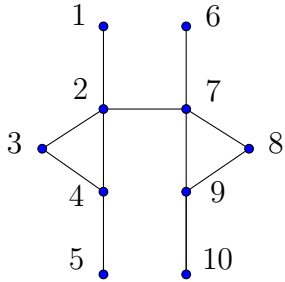
\begin{figure}[H]
\centering
\begin{tikzpicture}[scale=1.1,line cap=round,line join=round,>=triangle 45,x=1.0cm,y=1.0cm]
\draw (-1.,2.)-- (-1.,1.);
\draw (-1.,1.)-- (-1.,0.);
\draw (-1.,0.)-- (-1.,-1.);
\draw (-1.,1.)-- (-1.74,0.52);
\draw (-1.74,0.52)-- (-1.,0.);
\draw (0.76,0.52)-- (0.,1.);
\draw (0.76,0.52)-- (0.,0.);
\draw (0.,1.)-- (-1.,1.);
\draw (0.,0.)-- (0.,-1.);
\draw (0.,0.)-- (0.,-1.);
\draw (0.,2.)-- (0.,1.);
\draw (0.,1.)-- (0.,0.);
\draw (-1.52,2.4) node[anchor=north west] {$1$};
\draw (-1.5,1.5) node[anchor=north west] {$2$};
\draw (-2.3,0.85) node[anchor=north west] {$3$};
\draw (-1.55,0.2) node[anchor=north west] {$4$};
\draw (-1.55,-0.56) node[anchor=north west] {$5$};
\draw (0.1,2.38) node[anchor=north west] {$6$};
\draw (0.1,1.48) node[anchor=north west] {$7$};
\draw (0.86,0.85) node[anchor=north west] {$8$};
\draw (0.16,0.25) node[anchor=north west] {$9$};
\draw (0.06,-0.56) node[anchor=north west] {$10$};
\begin{scriptsize}
\draw [fill=blue] (-1.,2.) circle (1.5pt);
\draw [fill=blue] (-1.,1.) circle (1.5pt);
\draw [fill=blue] (-1.,0.) circle (1.5pt);
\draw [fill=blue] (-1.,-1.) circle (1.5pt);
\draw [fill=blue] (-1.74,0.52) circle (1.5pt);
\draw [fill=blue] (0.,1.) circle (1.5pt);
\draw [fill=blue] (0.,2.) circle (1.5pt);
\draw [fill=blue] (0.,0.) circle (1.5pt);
\draw [fill=blue] (0.,-1.) circle (1.5pt);
\draw [fill=blue] (0.76,0.52) circle (1.5pt);
\end{scriptsize}
\end{tikzpicture}
\vspace{.25cm}\caption{A graph $G$ for which $\mathcal{A}_G$ is a proper subposet of $\mathcal{Q}_{G}$.}
\label{Sa}
\end{figure}
 
}
\end{Example}
Now, to complete our discussion, we give an example of a graph $G$ for which $\mathcal{Q}_G=\mathcal{A}_G$. This example was appeared in \cite[Example~3.2]{A}.
\begin{Example}
\em{Let $G$ be the path on $5$ vertices illustrated in Figure \ref{path}. Thus, the minimal prime ideals of $J_G$ are:
$$P_{\emptyset}(G), P_{\{2\}}(G), P_{\{3\}}(G), P_{\{4\}}(G), P_{\{2,4\}}(G).$$ 
So, it was mentioned in \cite[Example~3.2]{A} that the elements of the poset $\mathcal{P}_G$ are:
\noindent \medskip
\newcommand{\splitatcommas}[1]{%
  \begingroup
  \ifnum\mathcode`,="8000
  \else
    \begingroup\lccode`~=`, \lowercase{\endgroup
      \edef~{\mathchar\the\mathcode`, \penalty0 \noexpand\hspace{0pt plus 1em}}%
    }\mathcode`,="8000
  \fi
  #1%
  \endgroup
}

\newcommand{\tuple}[1]{(\splitatcommas{#1})}
\newcommand{\set}[1]{\{\splitatcommas{#1}\}}

\setlength{\lineskiplimit}{2pt}\setlength{\lineskip}{3pt} 

$\splitatcommas{
 P_{\emptyset}(G), P_{\{2\}}(G), P_{\{3\}}(G), P_{\{4\}}(G), P_{\{2,4\}}(G), 
(x_2,y_2,f_{13},f_{14},f_{15},f_{34},f_{35},f_{45}), 
(x_3,y_3,f_{12},f_{14},f_{15},f_{24},f_{25},f_{45}), 
(x_2,y_2,x_3,y_3,f_{45}), 
(x_2,y_2,x_4,y_4,f_{13}),
(x_4,y_4,f_{12},f_{13},f_{15},f_{23},f_{25},f_{35}),
(x_2,y_2,x_4,y_4,f_{35}), 
(x_3,y_3,x_4,y_4,f_{12}), 
{\color{blue}(x_2,y_2,x_4,y_4,f_{13},f_{35})},
(x_2,y_2,x_3,y_3,x_4,y_4), 
(x_2,y_2,x_3,y_3,f_{14},f_{15},f_{45}),
(x_2,y_2,x_4,y_4,f_{13},f_{15},f_{35}), 
(x_3,y_3,x_4,y_4,f_{12},f_{15},f_{25}), 
(x_2,y_2,x_3,y_3,x_4,y_4,f_{15}).
}$ 

\par \medskip Note that $I={\color{blue}(x_2,y_2,x_4,y_4,f_{13},f_{35})}$, is the only non-prime ideal in the poset $\mathcal{P}_G$. Moreover, $I=q_1\cap q_2$ is the minimal primary decomposition for $I$, where
$$q_1=(x_2,y_2,x_4,y_4,f_{13},f_{15},f_{35})$$
and
$$q_2=(x_2,y_2,x_3,y_3,x_4,y_4).$$
Now, according to the construction of the poset $\mathcal{Q}_G$, we need to consider the following type of decomposition for $I$:
\begin{equation}\label{trick}
I=q_1\cap q_2\cap (q_1+P_\emptyset(G))\cap (q_2+P_\emptyset(G)).
\end{equation}
On the other hand, since $q_1+P_\emptyset(G)=q_1$ and $q_2+P_\emptyset(G)=q_1+q_2$, 
the elements of the poset $\mathcal{P}_I$, arised from the decomposition in $\eqref{trick}$ are exactly $q_1$, $q_2$ and $q_1+q_2$. Therefore, the constructions of the posets $\mathcal{Q}_G$ and $\mathcal{A}_G$ imply that 
$\mathcal{Q}_G=\mathcal{A}_G$.
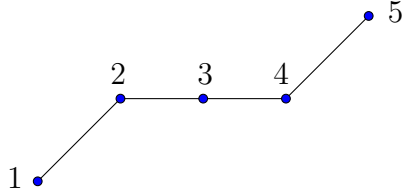
\begin{figure}[H]
\centering
\begin{tikzpicture}[scale=1.1,line cap=round,line join=round,>=triangle 45,x=1.0cm,y=1.0cm]
\draw (-2.,-1.)-- (-1.,0.);
\draw (-1.,0.)-- (0.,0.);
\draw (0.,0.)-- (1.,0.);
\draw (1.,0.)-- (2.,1.);
\draw (-2.5,-0.7) node[anchor=north west] {$1$};
\draw (-1.25,0.55) node[anchor=north west] {$2$};
\draw (-0.2,0.55) node[anchor=north west] {$3$};
\draw (0.72,0.55) node[anchor=north west] {$4$};
\draw (2.1,1.3) node[anchor=north west] {$5$};
\begin{scriptsize}
\draw [fill=blue] (-2.,-1.) circle (1.5pt);
\draw [fill=blue] (-1.,0.) circle (1.5pt);
\draw [fill=blue] (0.,0.) circle (1.5pt);
\draw [fill=blue] (1.,0.) circle (1.5pt);
\draw [fill=blue] (2.,1.) circle (1.5pt);
\end{scriptsize}
\end{tikzpicture}
\vspace{.25cm}\caption{A graph $G$ for which $\mathcal{Q}_{G}=\mathcal{A}_G$.}
\label{path}
\end{figure}

}
\end{Example}
\par Now, we are going to state the Hochster type formula for the local cohomology modules of binomial edge ideals arised from \cite[Theorem~3.9]{A}. First, we need to present the following proposition:
\par \medskip Here, $\Min(J)$, denotes the set of minimal prime ideals of an ideal $J$.
\begin{Proposition}
Let $G_i$ be a graph on $[n]$ and $T_i\subseteq [n]$ for each $1\leq i\leq k$. Let $J=\sum_{i=1}^{k}P_{T_i}(G_i)$ and $q\in \Min(J)$. Then $q=P_T(H)$, for some graph $H$ on $[n]$ and some $T\subseteq [n]$.
\begin{proof}
For each $1\leq i\leq k$, we have $P_{T_i}(G_i)=(x_s,y_s:s\in T_i)+J_{\widetilde{G}_{i1}}+\cdots+J_{\widetilde{G}_{ic_{G_i}(T_i)}}$. Let $H_1,\ldots,H_\ell$ be the connected components of the graph $G=(\bigcup_{i=1}^{k}\bigcup_{j=1}^{c_{G_i}(T_i)}\widetilde{G_{ij}})-\bigcup_{i=1}^{k}T_i$. One could easily see that
\begin{equation}\label{hard}
J=(x_s,y_s:s\in \bigcup_{i=1}^{k}T_i)+J_{H_1}+\cdots+J_{H_\ell}.
\end{equation} 
Now, by \cite[Problem~7.8,~part~(ii)]{HHO}, there exist $U_1,\ldots,U_{\ell}$ with ${U_i}\in \mathcal{C}(H_i)$ for each $1\leq i\leq \ell$, such that $q=(x_s,y_s:s\in \bigcup_{i=1}^{k}T_i)+\sum_{i=1}^{\ell}P_{U_i}(H_i)$. Let $T=(\bigcup_{i=1}^{k}T_i)\cup(\bigcup_{i=1}^{\ell}U_i)$. Now we take $H$ to be the graph which is obtained from the graph $L=(\bigcup_{i=1}^{\ell}\bigcup_{j=1}^{c_{H_i}(U_i)}\widetilde{H_{ij}})$, by adding those elements of $[n]$ which do not belong to the vertex set of the graph $L$, as isolated vertices. Then it follows that $q=P_T(H)$.
\end{proof} 
\end{Proposition}
Note that the above proposition together with the construction of the poset $\mathcal{Q}_G$ imply the next corollary that will be crucial throughout the paper:
\begin{Corollary}\label{crucial}
Let $G$ be a graph on $[n]$. Then, every element $q$ in the poset $\mathcal{Q}_G$ is of the form $P_T(H)$, for some graph $H$ on $[n]$ and some $T\subseteq [n]$.
\end{Corollary}
Before stating the Hochster type formula, we need to fix some notation:
\par Let $1_{{\mathcal{Q}_G}}$ be a terminal element that we add to the poset $\mathcal{Q}_G$. Then for every $q\in \mathcal{Q}_G$, by the interval $(q,1_{\mathcal{Q}_G})$, we mean the subposet
\[
\{z\in \mathcal{Q}_G:
 q\precneqq z\precneqq 1_{{\mathcal{Q}_G}}\}
\]
of the poset $\mathcal{Q}_G$.
\par \medskip Now, we are ready to state the Hochster type formula for the local cohomology modules of binomial edge ideals based on \cite[Theorem~3.9]{A}. Moreover, we would like to mention that since the poset $\mathcal{Q}_G$ which we consider in this paper is different from the poset considered by \`Alvarez~Montaner in \cite{A}, we provide a proof for this formula. However, the proof is similar to the one that was proposed in \cite[Theorem~3.9]{A}.
\begin{Theorem}\label{Hochster}
Let $G$ be a graph on $[n]$ and $\mathcal{Q}_G$ be the poset associated to $J_G$. Then we have the $\KK$-isomorphism
\[
H_{\mathfrak{m}}^{i}(S/J_G)\cong\bigoplus_{q \in \mathcal{Q}_G} H_{\mathfrak{m}}^{d_q}(S/q)^{\oplus M_{i,q}},
\]
where $d_q=\dim S/q$ and $M_{i,q}=\dim_{\KK} \widetilde{H}^{i-d_q-1}((q,1_{{\mathcal{Q}_G}});\KK)$.
\end{Theorem}
\begin{proof}
Let $q\in \mathcal{Q}_G$. By Corollary \ref{crucial}, we have that $q=P_T(H)$, for some graph $H$ on $[n]$ and some $T\subseteq [n]$. Now, the same method that was used for the proof of \cite[Proposition~3.7]{A} implies that $\mathcal{Q}_{G}$ is a subset of a distributive lattice of ideals of $S$. Moreover, since $S/q \cong \bigotimes_{i=1}^{c_H(T)}S_i/J_{\widetilde{H_i}}$, where $S_i=\KK[x_j,y_j:j\in V(H_i)]$, we have that $S/q$ is a Cohen-Macaulay domain, by \cite[Theorem~2.2]{A}. Therefore, the poset $\mathcal{Q}_G$ fulfills all the required conditions in \cite[Theorem~2.4]{A}. Now, since the minimality of the decompositions of non-prime ideals $I$ appear in Definition \ref{typo} does not matter in \cite[Theorem~5.22]{ABZ}, the result follows by \cite[Theorem~2.4]{A}.
\end{proof}	
\section{Topology of the subposets of the poset associated to binomial edge ideals}\label{topology}
	
In this section we investigate some topological properties of some specific subposets of the poset $\mathcal{Q}_G$ associated to the binomial edge ideal of a graph $G$.
\par The following lemma plays a vital role in our proofs.
\begin{Lemma}\label{vital}
Let $G$ be a graph on $[n]$. Then $q+P_\emptyset(G)\in \mathcal{Q}_G$, for every $q\in \mathcal{Q}_G$.
\end{Lemma}	
\begin{proof}
By the construction of the poset $\mathcal{Q}_G$, it is enough to show that $q+P_{\emptyset}(G)$ is a prime ideal. By Corollary \ref{crucial}, we have that $q=P_{T}(H)$ for some graph $H$ on $[n]$ and some $T\subseteq [n]$.  First we assume that $G$ is connected. Therefore, $q+P_{\emptyset}(G)= (x_i,y_i:i\in T) +J_{K_{n}-T}$, where $K_{n}-T$ denotes the complete graph on $[n]\backslash T$. Therefore, $q+P_{\emptyset}(G)=P_T(K_n)$, which implies that $q+P_\emptyset(G)$ is prime.
\par Next assume that $G$ is a disconnected graph with the connected components $G_1,\ldots,G_r$. We claim that every connected component of the graph $H$ is contained in the graph $G_i$, for some $1\leq i\leq r$. From this claim it will then follow that $q+P_{\emptyset}(G)= (x_i,y_i:i\in T) +\sum _{i=1}^{r} J_{K_{n_i}-T}$, where $n_i=|V(G_i)|$, for every $1\leq i\leq r$. Then $q+P_{\emptyset}(G)= P_T(\cup_{i=1}^{r}K_{n_i})$, which is a prime ideal. 
\par To prove the claim, by the construction of the poset $\mathcal{Q}_G$ and also by virtue of Corollary \ref{crucial} and using \eqref{hard} repeatedly, we may assume that $q\in \Min(J_1+\cdots+J_\ell)$, where $J_i\in \Min(J_G)$, for every $1\leq i\leq \ell$. On the other hand, by \cite[Problem~7.8,~part~(ii)]{HHO}, we have that $J_i=\sum_{j=1}^{r}P_{T_{ij}}(G_j)$ for each $1\leq i\leq \ell$, where $T_{ij}\in \mathcal{C}(G_j)$, for every $1\leq j\leq r$. Therefore, $\sum_{i=1}^{\ell}J_i=\sum_{j=1}^{r}\sum_{i=1}^{\ell}P_{T_{ij}}(G_j)$. This, together with \cite[Problem~7.8,~part~(ii)]{HHO}, Corollary \ref{crucial} and \eqref{hard} imply that the connected components of the graph $H$ should be contained in the connected components of the graph $G$, and hence the claim follows.

\end{proof}

\par The following definition is devoted to recall the concept of meet-contractibility for posets.

\begin{Definition}
\em{A poset $\mathcal{P}$ is said to be \emph{meet-contractible} if there exists an element $\alpha\in \mathcal{P}$ such that $\alpha$ has a meet with every element $\beta\in \mathcal{P}$.}
\end{Definition}
We say that a poset is \emph{contractible} if its order complex is contractible.
\par \medskip We use the following lemma to study the topology of some subposets of the poset $\mathcal{Q}_G$.
\begin{Lemma}\em{(\cite[Theorem~3.2]{BW}, see also \cite[Proposition~2.4]{W})}\label{meet}
\em{Every meet-contractible poset is contractible.}
\end{Lemma}
The following theorem is the main theorem of this section.
\begin{Theorem}\label{north}
Let $G$ be a graph on $[n]$ and assume that $\mathfrak{m} \in \mathcal{Q}_G$. Then $(\mathfrak{m},1_{\mathcal{Q}_G})$ is a contractible poset.
\end{Theorem}
\begin{proof}
Let $\mathcal{P}=(\mathfrak{m},1_{\mathcal{Q}_G})$. By Lemma \ref{meet}, it is enough to show that $\mathcal{P}$ is a meet-contractible poset.
\par Clearly, $P_\emptyset(G)\in \mathcal{P}$. Let $q\in \mathcal{P}$. We show that $P_\emptyset(G)$ and $q$ have a meet in $\mathcal{P}$. By Lemma \ref{vital}, we have $q+P_{\emptyset}(G)\in \mathcal{Q}_G$. On the other hand, by Corollary \ref{crucial} $q+P_\emptyset(G)\subsetneqq \mathfrak{m}$, since $q\subsetneqq \mathfrak{m}$ and $P_\emptyset(G)$ does not contain any variable. This implies that $q+P_{\emptyset}(G)\in \mathcal{P}$. 
\par Now we claim that $q+P_{\emptyset}(G)$ is the meet of $q$ and $P_\emptyset(G)$. Clearly, $q+P_{\emptyset}(G)\preccurlyeq q$ and $q+P_{\emptyset}(G)\preccurlyeq P_{\emptyset}(G)$. Now suppose that there exists $q^{'}\in \mathcal{P}$ with $q^{'}\preccurlyeq q$ and $q^{'}\preccurlyeq P_\emptyset (G)$. So, $q^{'}\preccurlyeq q+P_{\emptyset}(G)$. This means  that $q+P_\emptyset(G)$ is the meet of $q$ and $P_\emptyset(G)$, and hence the claim follows. Therefore, $\mathcal{P}$ is a meet-contractible poset.
\end{proof}
\begin{Remark}\label{remark}
\em{Let $\widehat{\mathfrak{m}}=P_T(H)$, where $H$ is an arbitrary graph on $[n]$ with $|T|=n-1$, where $n\geq 2$. Then with the same argument that we used in the proof of Theorem~\ref{north}, one checks that the result of Theorem \ref{north} still holds, if we replace $\widehat{\mathfrak{m}}$ with $\mathfrak{m}$.}
\end{Remark}
Now, as a consequence of Theorem \ref{north} and the above remark, we get the following corollary which will be used in the next section.
\begin{Corollary}\label{sug1}
Let $G$ be a graph on $[n]$ with $n\geq 2$. Let $q$ be an element of the poset $\mathcal{Q}_{G}$ such that $q\in \{ \mathfrak{m},\widehat{\mathfrak{m}}\}$. Then we have  $M_{i,q}=0$, for every $i$. 
\end{Corollary}
\section{Characterization of binomial edge ideals of small depth}\label{char}
In this section, as an application of the results provided in the previous section, we characterize all binomial edge ideals $J_G$, for which $\depth S/J_G=4$. First, we state the the following remark that enables us to  simplify the proofs in this section.
\begin{Remark}\label{sug2}
\em{Let $G$ be a graph on $[n]$. Then $d_q\neq 1$, for every $q\in \mathcal{Q}_G$. Indeed, suppose on the contrary that $d_q=1$. By Corollary \ref{crucial} we have that $q=P_T(H)$, for some $T\subseteq [n]$ and some graph $H$ on $[n]$. Now $d_q=1$ implies that either $|T|=n$ and $c_H(T)=1$, or $|T|=n-1$ and $c_H(T)=0$, where both of them are clearly impossible.}
\end{Remark}
Now, we supply a lower bound for the depth of binomial edge ideals that will be used later in our characterization.
\begin{Theorem}\label{lower}
Let $G$ be a graph on $[n]$. Then 
\[\depth S/J_G\geq 4r+\sum_{i=1}^{n}r_i(i+1),\]
where $r$ is the number of non-complete connected components of $G$ and $r_i$ is the number of complete connected components of $G$ of size $i$, for every $1\leq i\leq n$.
\end{Theorem}
\begin{proof}
Let $G_1,\ldots,G_\omega$ be the connected components of $G$. So, $S/J_G\cong \bigotimes_{i=1}^{\omega}S_i/J_{G_i}$, where $S_i=\KK[x_j,y_j:j\in V(G_i)]$. Therefore, $\depth S/J_G=\sum_{i=1}^{\omega}\depth S_i/J_{G_i}$. Then, by \cite[Theorem~1.1]{EHH}, the result follows if we show that $\depth S/J_G\geq 4$, whenever $n\geq 3$. 
\par Now by the definition of depth, it suffices to show that 
$H_{\mathfrak{m}}^{i}(S/J_G)=0$, for all $i$ with $0\leq i \leq 3$. Let $q\in \mathcal{Q}_G$, $d_q=\dim S/q$ and $M_{i,q}=\dim_{\KK} \widetilde{H}^{i-d_q-1}((q,1_{{\mathcal{Q}_G}});\KK)$. By Theorem \ref{Hochster}, it is enough to show that $M_{i,q}=0$, for $i=0,1,2,3$. 
\par First of all, Corollary \ref{crucial} implies that $q=P_{T}(H)$, for some graph $H$ on $[n]$ and some $T\subseteq [n]$. Now, we consider the following cases:
\par Let $i=0$. If $d_q > 0$, the assertion is clear. So we assume that $d_q=0$. We have $\mathrm{height}\hspace{0.9mm}P_T(H)=n-c_H(T)+|T|=2n$. This implies that $|T|-c_H(T)=n$. So~that $q=\mathfrak{m}$, since $|T|=n$ and $c_H(T)=0$. Now the result follows, since the order complex of the poset $(q,1_{\mathcal{Q}_G})$ is not empty. 
\par Let $i=1$. If $d_q\geq 2$, then the assertion is obvious. So, by Remark \ref{sug2} we may assume that $d_q=0$. Then we get $q=\mathfrak{m}$, and hence the result follows by Corollary~\ref{sug1}.
\par Let $i=2$. If $d_q\geq 3$, then the assertion is clear. In addition,  in Remark \ref{sug2} we showed that $d_q\neq 1$. So assume that 
$d_q\in\{0,2\}$. If $d_q=0$, then $q=\mathfrak{m}$. So that the result follows again by Corollary \ref{sug1}. Next suppose that $d_q=2$. Then we have $|T|-c_H(T)=n-2$. This implies that $|T|=n-1$ and $c_H(T)=1$. Therefore, $P_\emptyset(G)\in(q,1_{\mathcal{Q}_G})$, and hence the result follows since the order complex of the poset $(q,1_{\mathcal{Q}_G})$ is non-empty.
\par Let $i=3$. If $d_q\geq 4$, then the assertion holds.  Moreover, the discussion of the cases $d_q=0,1,2$ is similar to the previous cases. So we only need to consider the case $d_q=3$. In this case we have that $|T|=n-2$ and $c_H(T)=1$. So without loss of generality we may assume that 
\[q=(x_1,\ldots,x_{n-2},y_1,\ldots,y_{n-2})+(x_{n-1}y_{n}-x_{n}y_{n-1}).\] 
Now, we have that $P_\emptyset(G)\subsetneqq q$, since $n\geq 3$. Therefore, the order complex of the poset $(q,1_{\mathcal{Q}_G})$ is not empty, and hence the desired result follows.
\end{proof}
Note that the aforementioned lower bound in Theorem \ref{lower}, recovers the stated bound in \eqref{general} in Section \ref{intro}. Indeed, by the notation that we used in Theorem \ref{lower}, and by putting $t=\sum_{i=1}^{n}r_i$, it is not difficult to see that 
\[
\Bigl\lfloor \dfrac{2n-1}{\mathrm{bigheight}\hspace{0.9mm}J_G}\Bigr\rfloor
\leq r+t+1\leq 4r+2t\leq 4r+\sum_{i=1}^{n}r_i(i+1),\]
for every graph $G$ with at least a non-complete connected component. Moreover, we would like to remark that the given lower bound in Theorem \ref{lower} is sharp. For example, consider $G$ to be the graph depicted in Figure~\ref{sharp}. Then we have $\depth S/J_G=4$, by \cite[Theorem~4.4]{KumarS}.
\begin{figure}[H]
\centering
\begin{tikzpicture}[scale=1.1,line cap=round,line join=round,>=triangle 45,x=1.0cm,y=1.0cm]
\draw (-0.5,-1.)-- (-0.72,0.42);
\draw (-0.5,-1.)-- (0.72,0.42);
\draw (-0.5,-1.)-- (1.,0.);
\draw (-0.5,-1.)-- (-1.,0.);
\draw (0.5,-1.)-- (1.,0.);
\draw (0.5,-1.)-- (0.72,0.42);
\draw (0.5,-1.)-- (-0.72,0.42);
\draw (0.5,-1.)-- (-1.,0.);
\draw (0.,0.78)-- (-0.5,-1.);
\draw (0.,0.78)-- (0.5,-1.);
\begin{scriptsize}
\draw [fill=blue] (-0.5,-1.) circle (1.5pt);
\draw [fill=blue] (0.5,-1.) circle (1.5pt);
\draw [fill=blue] (1.,0.) circle (1.5pt);
\draw [fill=blue] (-1.,0.) circle (1.5pt);
\draw [fill=blue] (-0.72,0.42) circle (1.5pt);
\draw [fill=blue] (0.72,0.42) circle (1.5pt);
\draw [fill=blue] (0.,0.78) circle (1.5pt);
\end{scriptsize}
\end{tikzpicture}
\vspace{.25cm}\caption{A complete bipartite graph $G$ with $\depth S/J_G=4$.}
\label{sharp}
\end{figure}
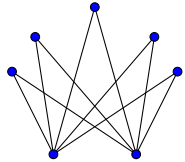

\par \medskip Now we are ready to state our main theorem. Here, we denote by $2K_1$, the graph consisting of two isolated vertices.
\begin{Theorem}\label{small}
Let $G$ be a graph on $[n]$ with $n\geq 4$. Then the following are equivalent:
\begin{enumerate}
\item[{(a)}] $\depth S/J_G=4$.
\item[{(b)}] $G=G'\ast 2K_1$, for some graph $G'$.
\end{enumerate}  
\end{Theorem}
\begin{proof}
$(a)\Rightarrow (b)$: Assume that $G\neq G'\ast 2K_1$, for any graph $G'$. We show that depth~$S/J_G\geq 5$. By the definition of depth and by Theorem  \ref{lower}, it suffices to show that $H_{\mathfrak{m}}^{4}(S/J_G)=0$. 
\par We keep using the notation that we used in Theorem \ref{lower}. Let $q\in \mathcal{Q}_G$, $d_q=\dim S/q$ and $M_{4,q}=\dim_{\KK} \widetilde{H}^{3-d_q}((q,1_{{\mathcal{Q}_G}});\KK)$. By Theorem \ref{Hochster}, the result follows once we show that $M_{4,q}=0$. Notice that Corollary \ref{crucial} implies that $q=P_{T}(H)$, for some graph $H$ on $[n]$ and some $T\subseteq [n]$. Note also that if $d_q\geq 5$, then the assertion is obvious. On the other hand, as we discussed in Remark \ref{sug2}, we have $d_q\neq 1$. So we consider the following cases:
\par Let $d_q=0$. So $q=\mathfrak{m}$, since $|T|=n$ and $c_H(T)=0$. Therefore, $M_{4,q}=0$, by Corollary \ref{sug1}.
\par Let $d_q=2$. We have $|T|-c_H(T)=n-2$. This implies that $|T|=n-1$ and $c_H(T)=1$. Now the result follows by Corollary \ref{sug1}.
\par Let $d_q=3$. We need to show that the order complex of the poset $(q,1_{\mathcal{Q}_G})$ is connected. Note that since $\mathrm{height}\hspace{0.9mm}q=n-c_H(T)+|T|=2n-3$, we have $|T|=n-2$ and $c_H(T)=1$, and hence we may assume that
\[q=(x_1,\ldots,x_{n-2},y_1,\ldots,y_{n-2})+(x_{n-1}y_{n}-x_{n}y_{n-1}).\] 
\par Now suppose that $q_{1}, q_{2}\in (q,1_{\mathcal{Q}_G})$ and $q_1\neq q_2$. By Corollary \ref{crucial}, we have that $q_1=P_{T_1}(H_1)$ and $q_2=P_{T_2}(H_2)$, for some graphs $H_1$ and $H_2$ on $[n]$ and some $T_1,T_2\subseteq [n]$. Moreover, we have that $T_1,T_2\subseteq \{1,\ldots,n-2\}$, since $q_{1}, q_{2}\in (q,1_{\mathcal{Q}_G})$. Now without loss of generality we may assume that $T_1\subseteq \{1,\ldots,n-2\}$ and $T_2\subsetneqq \{1,\ldots,n-2\}$, since $q_1\neq q_2$. We first assume that $T_1\subsetneqq \{1,\ldots,n-2\}$. Thus we have $q_{1}+P_{\emptyset}(G)\in (q,1_{\mathcal{Q}_G})$ and $q_{2}+P_{\emptyset}(G)\in (q,1_{\mathcal{Q}_G})$, by Lemma \ref{vital}. So we get the path
\[q_{1}, q_{1}+P_{\emptyset}(G), P_{\emptyset}(G), q_{2}+P_{\emptyset}(G), q_{2}\]
in the 1-skeleton graph of the order complex of the poset $(q,1_{\mathcal{Q}_G})$.
\par Next assume that $T_1=\{1,\ldots,n-2\}$. Now since the set of maximal elements of the poset $\mathcal{Q}_G$ coincides with the set of minimal prime ideals of $J_G$, there exists $U\in \mathcal{C}(G)$ such that 
$P_U(G)\subseteq q_1$. We show that $U\subsetneqq \{1,\ldots,n-2\}$. Indeed, otherwise by \cite[Corollary~3.9]{HHHKR} we get $G=L\ast 2K_1$, where $L=G-\{n-1,n\}$. This contradicts the fact that $G\neq G'\ast 2K_1$, for any graph $G'$. So, $P_U(G)\subsetneqq q_1$. On the other hand, $P_U(G)+P_{\emptyset}(G)\in (q,1_{\mathcal{Q}_G})$ by Lemma \ref{vital}. Therefore, we get the path 
\[
q_{1}, P_U(G), P_U(G)+P_{\emptyset}(G), P_{\emptyset}(G), q_{2}+P_{\emptyset}(G), q_{2}\]
in the 1-skeleton graph of the order complex of the poset $(q,1_{\mathcal{Q}_G})$.
\par Therefore, the order complex of the poset $(q,1_{\mathcal{Q}_G})$ is connected, as desired.
\par Now let $d_q=4$. Therefore, $|T|-c_H(T)=n-4$. So, the following cases occur:
\par First assume that $|T|=n-2$ and $c_H(T)=2$. Therefore, we may assume that  $q=(x_1,\ldots,x_{n-2},y_1,\ldots,y_{n-2})$. Then by a similar method that we used in the last part of the case $d_q=3$, there exists $W\in \mathcal{C}(G)$ such that $P_W(G)\subsetneqq q$. This means that $P_W(G)\in (q,1_{\mathcal{Q}_G})$, and hence $(q,1_{\mathcal{Q}_G})$ is non-empty.
\par Next assume that $|T|=n-3$ and $c_H(T)=1$. So, we may assume that 
\[q=(x_1,\ldots,x_{n-3},y_1,\ldots,y_{n-3},x_{n-2}y_{n-1}-x_{n-1}y_{n-2},
x_{n-2}y_{n}-x_{n}y_{n-2},
x_{n-1}y_{n}-x_{n}y_{n-1}).
\]
Now, we have $P_\emptyset(G)\in (q,1_{\mathcal{Q}_G})$, since $n\geq 4$. Thus, $(q,1_{\mathcal{Q}_G})$ is non-empty.
\par Therefore, $M_{4,q}=\dim_{\KK} \widetilde{H}^{-1}((q,1_{\mathcal{Q}_G});\KK)=0$.
\par \medskip $(b)\Rightarrow (a)$: Assume that $G=G^{'}\ast 2K_1$, for some graph $G^{'}$. Now, if $G'$ is a complete graph then, the result follows from \cite[Theorem~3.9]{KumarS}. So, assume that $G'$ is not complete. Therefore, by Theorem \ref{lower}, we have $\depth S'/J_{G'}\geq 4$, where $S'=\KK[x_i,y_i:i\in V(G')]$. This, together with \cite[Theorem~4.3]{KumarS} and \cite[Theorem~4.4]{KumarS}, imply the result.
\end{proof}

\vspace{1cm}
\par \textbf{Acknowledgments:} The authors would like to thank Josep \`Alvarez Montaner and the anonymous referee for their useful comments. The authors would also like to thank Institute for Research in Fundamental Sciences (IPM) for financial support. The research of the second author was in part supported by a grant from IPM (No. 99130013). The research of the third author was in part supported by a grant from IPM (No. 99050212).

\end{document}